\documentclass[ a4paper,draft]{mcom-l}
\usepackage{amsmath, amsfonts, amsthm, amssymb, graphicx}

\usepackage{dcolumn}
\newcolumntype{d}[1]{D{.}{.}{#1}}

\newcommand     {\C}            {{\mathbb C}}
\newcommand     {\Mp}[1][\Z]    {\sym{Mp}(2,{#1})}
\newcommand     {\PH}           {{\mathbb H}}

\newcommand     {\SL}[1][\Z]    {\sym{SL}(2,{#1})}
\newcommand     {\Z}            {{\mathbb Z}}
\newcommand     {\leg}[2]       {\left(\frac {#1}{#2}\right)}
\newcommand     {\mat}[4]       {\left(\begin{smallmatrix} #1&#2 \\ #3&#4\end{smallmatrix}\right)}
\newcommand     {\psp}[2]       {\langle #1|#2 \rangle}
\newcommand     {\sym}[1]       {\operatorname{#1}}

\theoremstyle{plain}
\newtheorem*{Main Theorem}{Main Theorem}
\newtheorem*{Supplement}{Supplement to the Main Theorem}
\newtheorem{Theorem}{Theorem}
\newtheorem*{Lemma}{Lemma}
\theoremstyle{definition}
\newtheorem*{Remark}{Remark}

\title[Numerical Computation of a Certain Dirichlet Series]{%
  Numerical Computation of a Certain Dirichlet Series attached to
  Siegel Modular Forms of Degree~Two }

\author{%
  Nathan C. Ryan, Nils-Peter Skoruppa, Fredrik Str\"omberg
}

\address{%
  Department of Mathematics, Bucknell University\endgraf
  nathan.ryan@bucknell.edu\endgraf
  \null
  Fachbereich Mathematik, Universit\"at Siegen
  \endgraf
  nils.skoruppa@uni-siegen.de \endgraf \null Fachbereich Mathematik,
  TU-Darmstadt \endgraf stroemberg@mathematik.tu-darmstadt.de }

\thanks{%
  This project was supported by the National Science Foundation under
  FRG Grant No. DMS-0757627 and we also made use of hardware provided
  by DMS-0821725.
}


\begin{document}

\begin{abstract}
  The Rankin convolution type Dirichlet series $D_{F,G}(s)$ of Siegel
  modular forms $F$ and $G$ of degree two, which was introduced by
  Kohnen and the second author, is computed numerically for various
  $F$ and $G$.  In particular, we prove that the series $D_{F,G}(s)$,
  which share the same functional equation and analytic behavior with
  the spinor $L$-functions of eigenforms of the same weight are not linear
  combinations of those.  In order to conduct these experiments a
  numerical method to compute the Petersson scalar products of
  Jacobi Forms is developed and discussed in detail.
\end{abstract}

\maketitle

\section{Introduction and Statement of Results}
\label{sec:introduction}

For cusp forms $F$ and $G$ in the space $S_k$ of Siegel cusp forms of
degree 2 and even weight $k$ on the full Siegel modular group, let
$D_{F,G}(s)$ be the Rankin convolution type Dirichlet series which was
introduced in~\cite{Kohnen-Skoruppa}.  It was proved~\cite[first
theorem]{Kohnen-Skoruppa} that $D_{F,G}(s)$ shares the same analytic
properties and the same functional equation as the spinor $L$-function
$L(F,s)$ associated to a Hecke eigenform $F$ in~$S_k$. In fact, if $G$
is in the spezialschaar then one has~\cite[second
theorem]{Kohnen-Skoruppa} $D_{F,G}(s) = \psp {F_1}{G_1} L(F,s)$ (see
Section~\ref{sec:background} for notations). However, it was left open
whether $D_{F,G}(s)$, for {\em interesting} Hecke
eigenforms\footnote{Following the terminology of~\cite{Skoruppa2} we
  call a cusp form {\em interesting} if it belongs to the orthogonal
  complement of the Maass spezialschaar.} $F,G$, is related to the
spinor $L$-functions of the forms in $S_k$. In any case it cannot in
general simply satisfy the same identity as before since, for example
$D_{F,F}(s)$, for $F$ not in the Maass spezialschaar, has a simple
pole at $s=k$ with residue equal to the Petersson norm of $F$, whereas
$L(F,s)$ is holomorphic.  In fact, a spinor $L$-function of a Hecke
eigenform in $S_k$ has a pole if and only if $F$ belongs to the Maass
spezialschaar~\cite{Evdokimov}.

The first interesting Hecke eigenform is in weight $20$ and the space
of such forms is 1-dimensional in weights 20 and 22, is 2-dimensional
in weights 24 and 26 (which, remarkably, have rational Fourier
coefficients and eigenvalues), is 3-dimensional in weight 28 and
4-dimensional in weight 30. The interesting eigenforms in weights 20
up to 26 have rational Fourier coefficients (and eigenvalues), whereas
the eigenforms in weight 28 and 30 are conjugate under Galois
conjugation, respectively. The aim of this article is twofold. We
report the results of numerical experiments which we performed to
study the series $D_{F,G}$ for interesting eigenforms $F$ and $G$ in
weights 20 up to 30. Secondly, we develop and describe the necessary
theoretical tools for our experiments. Necessary to the computations
which we report about out in this note is also the ability to compute
Siegel modular forms, their Hecke eigenvalues and their spinor
$L$-series.  We do this via a Sage~\cite{Sage} package written by the
first two authors~\cite{Ryan-Skoruppa} (with the aid of several
people) and based on the algorithms developped
in~\cite{Skoruppa2}. The code and the data which we used for our
computations can be found under~\cite{Data}. The interested reader may
use these data to reproduce our experiments.

Our computational results are summarized as follows.
\begin{Main Theorem}
  \label{thm:main}
  For $k = 20, 22, 24, 26, 28, 30$, let $F$ and $G$ denote
  simultaneous eigenforms of weight $k$ that are cusp forms but not in
  the Maass spezialschaar.  Then $D_{F,G}(s)$ is not a linear
  combination of the spinor $L$-functions $L(H,s)$ where~$H$ ranges
  over all simultaneous eigenforms (including non cusp forms) of
  weight $k$.
\end{Main Theorem}
This result is surprising since, as already pointed out, the series
$D_{F,G}(s)$ shares the same functional equation and $\Gamma$-factors
as $L(F,s)$ and $L(G,s)$ yet seems not to be related to the usual
spinor $L$-functions. Note that the theorem is proved by numerical
computations. However, we made careful error estimates to ensure the
validity of our result.

In fact, we actually prove more. Namely, we show in
section~\ref{sec:proof} that all the series $D_{F,G}(s)$ and $L(H,s)$,
where $\{F,G\}$ ranges over all subsets with one or two elements of
the interesting eigenforms of weight $k$, and where $H$ ranges over
all eigenforms of weight $k$, are linearly independent.

Investigating the (numerically approximated) coefficients of the
series $D_{F,G}(s)$ more closely we find:
\begin{Supplement}
  \label{prop:mult}
  In the notations of the Main Theorem the coefficients of
  $D_{F,G}(s)$ are not multiplicative; that is, the (normalized)
  coefficients $a(n)$ do not satisfy $a(m)a(n)=a(mn)$ for all coprime
  pairs of integers $m$ and $n$.

\end{Supplement}
We remark that this supplement implies that the series $D_{F,G}(s)$,
for $F$ and $G$ not is the Maass spezialschaar, is not an $L$-function
in the sense of Selberg, i.e., does not belong to the so-called
Selberg class.  Note also that this is the second kind of Dirichlet
series which is attached to a Siegel modular form, has a meromorphic
continuation, satisfies a functional equation but is not an
$L$-function in the sense of Selberg, the other being the
K\"ocher-Maass series.  It would be interesting to have a deeper
understanding of how these series $D_{F,G}(s)$ fit into the general
framework of automorphic forms and their associated $L$-functions.

In Section~\ref{sec:background} we summarize the theoretical framework
for our experiments.  In Section~\ref{sec:computation} we develop the
necessary tools for our computations and discuss computational issues.
In particular, in order to obtain our numerical results we have to
develop methods to compute Poincar\'e series and Petersson scalar
products of Jacobi forms. These methods are of independent interest;
they are described in detail in
Sections~\ref{sec:Petersson-scalar-products}
to~\ref{sec:Weil-representation}. In Section~\ref{sec:proof} we
explain the proof of the main theorem. Finally, in
section~\ref{sec:numerical-results} we summarize our numerical
results.

\section{Background and Notation}
\label{sec:background}

The Siegel upper half space of degree~$d$ is given by the set
\begin{equation*}
  \PH^d := \{Z=X+iY\in \C^{d \times\ d} : Z^t = Z, \text{ and $Y$
    positive definite}\}.
\end{equation*}
A Siegel modular form $F:\PH^d \to \C$ of degree~$d>1$ is a
holomorphic function invariant under the action of $\sym{Sp}(d,\Z)$ on
$\PH^d$. In this paper we specialize to degree~$2$ and so we have a
Fourier expansion of the form
\begin{equation*}
  F\mat \tau z z  {\tau'}
  =
  \sum_{\substack{r,n,m\in\Z\\r^2-4mn\leq 0\\n,m\geq 0}}
  a_F(n,r,m)q^n\zeta^rq'^m
  ,
\end{equation*}
where $q=e^{2\pi i \tau}$, $q'=e^{2\pi i \tau'}$ and $\zeta = e^{2\pi
  i z}$ ($\tau,\tau'\in \PH^1$, $z\in\C$).  Such an $F$ is a cusp form
if the $a_F(n,r,m)=0$ unless the quadratic form $[n,r,m] =
nx^2+rxy+my^2$ is strictly positive definite.  We denote by $M_k$ and
$S_k$ the space of Siegel modular forms and cusp forms of weight $k$
and degree 2, respectively.

For each integer $\ell \ge 1$ there is a Hecke operator $T(\ell)$, a
linear map that maps $S_k$ to~$S_k$ (see~\cite[p.~386]{Skoruppa2} for
explicit formulas).  We call a cusp form an {\it eigenform} if it is a
common eigenform for all Hecke operators. To eigenforms we attach the
spinor $L$-function as follows.  If $T(\ell)F=\lambda_{F,\ell}F$, then
define
\begin{equation}
  \label{eq:spinor-l-function}
  L(F,s):=\zeta(2s-2k+4)\sum_{\ell=1}^\infty \frac{\lambda_{F,\ell}}{\ell^s}
\end{equation}
This function has a meromorphic continuation to the whole complex
plane with at most a pole, which is then at $s=k$ and of order one.
It satisfies the functional equation
\begin{equation}
  \label{eq:functional}
  L^*(F,s):=(2\pi)^{-2s}\Gamma(s)\Gamma(s-k+2)L(F,s)=(-1)^kL^*(F,2k-2-s).
\end{equation}
In some of our computations we renormalize \eqref{eq:functional} by
$s\mapsto s+k-\frac32$ so that the corresponding new functional
equation expresses an invariance under $s\mapsto 1-s$ instead
of~$s\mapsto 2k-2-s$.

We use $S_{k,N}$ for the space of Jacobi cusp forms of weight $k$ and
index $N$ in the sense of~\cite[Chapter I]{Eichler-Zagier}. Any $\phi$
in $S_{k,N}$ has a Fourier development of the form
\begin{equation*}
  \phi(\tau,z) = \sum C_\phi(D,r)\,q^{\frac{r^2-D}{4N}}\,\zeta^r .
\end{equation*}
Here the sum is over all {\it $N$-admissible pairs}, i.e.~over all
pairs of integers $(D,r)$ such that $D<0$ and $D \equiv r^2 \bmod 4N$.
Recall that $C_\phi(D,r)$ depends on $r$ only modulo $2N$ and that
$\phi$ is a cusp form if $C_\phi(D,r)=0$ unless $D<0$

As mentioned above, for an integer $k$ we denote by $S_k$ the space of
Siegel cusp forms of weight~$k$ on the full Siegel modular group, and
we use $S_k^*$ and $S_k^?$ for, respectively, the subspace of Maass
cusp forms and its complement. Recall that $S_k^*$ and $S_k^?$ are
invariant under all Hecke operators. The first case of nontrivial
$S_k^?$ occurs for weight $k=20$, where $S_{20}^?$ is one-dimensional.
In Table~\ref{tbl:Siegel-dimensions} we list, for $20 \le k \le 30$,
the dimensions of various subspaces of the full space of Siegel
modular forms of degree 2 and weight $k$.

Every element $F$ in $S_k$ admits a Fourier-Jacobi development
\begin{equation*}
  F \mat \tau z z {\tau'} = \sum_{N\ge 1} F_N(\tau,z)\,e^{2\pi i N\tau'}
\end{equation*}
with elements $F_N$ of $S_{k,N}$, respectively.

If $F$ and $G$ are both elements of $S_k$ we consider the Dirichlet
series
\begin{equation}
  D_{F,G}(s) := \zeta(2s-2k+4)\sum_{N \ge 1} \frac{\psp {F_N}{G_N}}{N^s}
  ,
\end{equation}
which was introduced in~\cite{Kohnen-Skoruppa} and which may be viewed
as an analogue of the Rankin convolution for elliptic modular forms.
Here $\psp {\phi}{\psi}$, for Jacobi cusp forms $\phi$ and $\psi$ in
$S_{k,N}$, denotes their Petersson scalar product. More precisely, we
have
\begin{equation*}
  \psp {\phi}{\psi} = \int_{\mathcal D}
  \phi(\tau,z)\overline{\psi(\tau,z)} \,v^{k-3}e^{-4\pi
    Ny^2/v}du\,dv\,dx\,dy \qquad(\tau = u+iv, z = x+iy),
\end{equation*}
where ${\mathcal D}$ denotes any fundamental domain for the Jacobi
group $\SL^{J}$ acting on $\PH \times \C$.

\section{Method of computation}
\label{sec:computation}

\subsection{Petersson scalar products}
\label{sec:Petersson-scalar-products}

In order to compute numerically the coefficients of the $L$-series
$D_{F,G}$ we need a method for computing the scalar product of two
Jacobi forms of the same weight $k$ and index $N$. Since $F$ and $G$
are eigenforms, their Fourier coefficients (and hence also the Fourier
coefficients of the Jacobi forms in their Fourier-Jacobi development)
are algebraic numbers. The package~\cite{Ryan-Skoruppa} implements
methods developed in~\cite{Skoruppa2} to provide these Fourier
coefficients in an exact representation. It is thus reasonable to look
for a method to express the Petersson scalar product of two Jacobi
forms in terms of their Fourier coefficients. We propose the following
method.

For an $N$-admissible pair $a = (D,r)$ denote by $P_a$ the $a$-th
Poincar\'e series of $S_{k,N}$.  This is the unique Jacobi cusp form
which satisfies
\begin{equation*}
  \psp \phi{P_a} = C_\phi(a)
\end{equation*}
for all forms $\phi$ in $S_{k,N}$. It is clear that the $P_a$ with $a$
running through all admissible pairs span the space of cusp forms. Let
$n$ be the dimension of $S_{k,N}$, and let $h=\{h(i)\}$ ($1\le i\le
n$) be a sequence of $n$ admissible pairs.  We use $\Gamma_h$ for the
Gram matrix of the $P_{h(i)}$ with respect to the Petersson scalar
product, i.e.
\begin{equation}
  \label{eq:Gram-matrix}
  \Gamma_h = \left(C_{P_{h(i)}}(h(j))\right).
\end{equation}
Note that the determinant of $\Gamma_h$ is different from 0 if and
only if the $P_{h(i)}$ form a basis of $S_{k,N}$.  For a form $\phi$
in $S_{k,N}$, let
\begin{equation*}
  \gamma_h(\phi) = \left(C_\phi(h(j)),\dots,C_f(h(n))\right)^t .
\end{equation*}
We then have the following formula for the scalar product of $F$ and
$G$.
\begin{Theorem}
  \label{thm:scalar-product}
  Assume that the Poincar\'e series $P_{h(i)}$ ($1\le i\le n$) form a
  basis of $S_{k,N}$.  Let $\phi$ and $\psi$ be Jacobi forms in
  $S_{k,N}$.  Then one has
  \begin{equation*}
    \psp {\phi}{\psi}
    =
    \gamma_h(\phi)^t\Gamma_h^{-1}\,\overline {\gamma_h(\psi)} .
  \end{equation*}
\end{Theorem}
\begin{proof}
  Indeed, if $x$ and $y$ are the coordinate vectors of $\phi$
  respectively $\psi$ with respect to the basis $P_{h(i)}$ we have
  $\psp \phi \psi = x^t \Gamma_h \overline y$.  On the other hand, we
  deduce from $\psp \phi {P_{h(i)}} = C_{\phi}(h(i))$ the identity $
  \overline{\Gamma_h}x = \gamma_h(\phi)$ and similarly
  $\overline{\Gamma_h}y = \gamma_h(\psi)$. Using $\overline{\Gamma_h}
  = \Gamma_h^t$ the lemma is now immediate.
\end{proof}

\subsection{Fourier coefficients of Jacobi Poincar\'e series}

According to Theorem~\ref{thm:scalar-product} the computation of the
$N$-th coefficient in the series $D_{F,G}$ is reduced to the numerical
computation of the Fourier coefficients of the Jacobi Poincar\'e
series $P_{\Delta,r}$ in $S_{k,N}$ where $\Delta<0$ and
$r^{2}\equiv\Delta\mod4N$. Explicit formulas for these coefficients
were originally worked out in~\cite[Prop.~in
Sect.~2]{Gross-Kohnen-Zagier}.  However, we shall use the formulas
from~\cite[Thm.~1.4]{Bruinier}, which are stated in terms of vector
valued modular forms instead of Jacobi forms.  The formulas given
in~\cite[Thm.~1.4]{Bruinier} are more useful for us mainly because
they involve in an explicit manner the Weil representation $\rho$
associated to the lattice $L:=\left(\mathbb{Z},(x,y)\mapsto
  2Nxy\right)$.  Explicit formulas for the matrix coefficients of this
representation have been given in~\cite{Skoruppa1} and will be given
for arbitrary lattices and in a slightly more explicit form
in~\cite{Stroemberg}. This makes it possible to evaluate the
coefficients of the Poincar\'e series without having to deal with too
many iterated sums\footnote{ A generalization and implementation of
  this method for arbitrary Weil representations will be published
  elsewhere~\cite{Stroemberg}.}  (more precisely, Gauss sums inside
Kloosterman sums) as provided by the formulas
in~\cite{Gross-Kohnen-Zagier}.

To be more specific, for our computations we shall use the following
formulas:
\begin{Theorem}
  \label{thm:Poincar-series}
  We have
  \begin{equation*}
    P_{\Delta,r}\left(z,\tau\right)
    =
    \alpha_{N,k,\Delta,r}
    \sum_{\begin{subarray}{c}
        r',\Delta'\in\Z, \Delta <0\\
        r'^{2} \equiv \Delta' \bmod4N
      \end{subarray}} p_{\Delta,r}\left(\Delta',r'\right)\,
    q^{\frac{r'^{2}-\Delta}{4N}}\zeta^{r'} ,
  \end{equation*}
  where the Fourier coefficients are given by
  \begin{eqnarray*}
    p_{\Delta,r}\left(\Delta',r'\right)
    & = &
    \delta_{\Delta,\Delta'}\left(\delta_{r,r'}+\delta_{-r,r'}\right)\nonumber \\
    & &
    +2\pi\left|\frac{\Delta'}
      {\Delta}\right|^{\frac{k}{2}-\frac{3}{4}}
    \sum_{c\in\Z, c \ne 0}
    H_{c}^{*}\left(r,\Delta,r',\Delta'\right)
    J_{k-\frac{3}{2}}\left(\frac{\pi}{cN}\sqrt{\left|\Delta\Delta'\right|}\right).
    \label{eq:p_D_r_Dp_rp}
  \end{eqnarray*}
  Here $J_{v}\left(x\right)$ is the standard J-Bessel function and
  \begin{equation*}
    H_{c}^{*}\left(r,\Delta,r',\Delta'\right)
    =
    \frac{e^{-\pi i\sym{sgn}(c) \frac{\left(2k-1\right)}{4}}}{|c|}
    \sum_{\begin{subarray}{c}
        d \bmod c\\
        \sym{gcd}(c,d)=1
      \end{subarray}
    }
    \rho
    \left(\mat abcd^*\right)_{\underline r, \underline r'}
    e\left(-\frac{\Delta
        a+\Delta'd}{4Nc}\right)
    ,
  \end{equation*}
  where $\rho$ is the Weil representation of $\SL$ associated to the
  lattice $\big(\Z, (x,y)\mapsto 2Nxy\big)$, and where, for an integer
  $r$, we use $\underline r$ for the element $\frac r{2N} +\Z$ in
  $\frac1{2N}\Z/\Z$ (see Section~\ref{sec:Weil-representation} for an
  explanation of these terms, and see
  Theorem~\ref{thm:Weil-representation} for an explicit expression for
  the matrix elements of $\rho$), and where, for each~$d$ relatively
  prime to $c$, we choose an element $\mat abcd$ in $\SL$.  Finally,
  one has
  \begin{equation*}
    \alpha_{N,k,\Delta,r}
    =
    \frac
    {(\pi|\Delta|/N)^{k-3/2}}
    {8\sqrt N \Gamma(k-3/2)}
    .
  \end{equation*}
\end{Theorem}
\begin{Remark}
  Note that the Fourier coefficients of the Poincar\'e series are real
  as follows, e.g., from~\cite[Lemma 1.13 and 1.15]{Bruinier}.
\end{Remark}
\begin{proof}[Proof of Theorem~\ref{thm:Poincar-series}]
  According to~\cite[Theorem~5.1]{Eichler-Zagier} we have a one to one
  correspondence between Jacobi forms and vector-valued modular
  forms. Namely, if $\phi$ is a Jacobi form in $J_{k,N}$ with Fourier
  coefficients $C_\phi(\Delta,r)$, then the application
  \begin{equation*}
    \phi \mapsto H(\phi)
    :=
    \sum_{\rho\bmod N} C_\phi(\Delta, \rho)\,e_{\rho/2N + \Z}
  \end{equation*}
  defines a bijection $H:J_{k,N} \rightarrow M_{k-\frac12,L}$.  Here,
  for a residue class $x$ in $\frac1{2N}\Z/\Z$, we use $e_{x}$ for the
  element in the group ring $\C[\frac1{2N}\Z/\Z]$ which is $1$ at $x$
  and $0$ otherwise. Moreover, following the notations
  in~\cite[Def.~1.2, p.18]{Bruinier}, we use $M_{k-\frac12,L}$ for the
  space of vector valued modular forms $f:\PH \rightarrow
  \C[\frac1{2N}\Z/\Z]$ which satisfy
  \begin{equation*}
    f(A\tau) = w(\tau)^{2k-1}\rho^*(A,w)f(\tau)
  \end{equation*}
  for all $(A,w) \in \Mp$ with $\Mp$ denoting the usual metaplectic
  cover of $\SL$ (see \cite[sec.~1.1, p.15]{Bruinier}). The
  representation $\rho$ is the Weil representation associated to the
  lattice $L=\left(\mathbb{Z},(x,y)\mapsto 2Nxy\right)$, i.e.~the
  representation of $\Mp$ on $\C[\frac1{2N}\Z/\Z]$, which in terms of
  the standard generators for $\Mp$ is given by
  \begin{equation}
    \label{eq:Weil-representation}
    \begin{split}
      \rho\big(\mat 1101,1\big)\,e_{r/2N+\Z} &= e^{2\pi i r^2/4N}\,e_{r/2N+\Z}\\
      \rho\big(\mat 0{-1}10,\sqrt\tau\big)\,e_{r/2N+\Z} &=
      \frac{e^{-2\pi i/8}}{\sqrt{2N}}\sum_{r'\bmod 2N}\,e^{-2\pi i
        rr'/2N}\,e_{r'/2N+\Z}
    \end{split}
  \end{equation}
  Finally, $\rho^*$ is the representation obtained from $\rho$ by
  taking the complex conjugates of the $\rho(A,w)$ (where we view
  $\rho(A,w)$ as a matrix).

  By~\cite[Thm.~5.3]{Eichler-Zagier} we have
  \begin{equation}
    \label{eq:psp-compatibility}
    \psp {\phi}{\psi} = \frac 1{4\sqrt{N}} \psp {H(\phi)}{H(\psi)}, 
  \end{equation}
  where the scalar product of two vector valued modular forms $f_1$
  and $f_2$ in $M_{k-\frac12,L}$ is defined as~\cite[eq.~(1.17),
  p.~22]{Bruinier}:
  \begin{equation*}
    \psp{f_1}{f_2}
    =
    \sum_{r \in \frac1{2N}\Z/\Z}
    \int_{\SL\backslash\PH} f_{1,j}(\tau)\overline{f_{2,j}(\tau)}v^{k-\frac52}dudv
    \quad
    (f_j =\sum_{r \in \frac1{2N}\Z/\Z} f_{j,r} e_r)
    .
  \end{equation*}
  (Note that~\cite[Thm.~5.3]{Eichler-Zagier} contains an error: the
  factor $1/{\sqrt{2m}}$ in Thm.~5.3 has to be replaced by
  $1/{\sqrt{4m}}$ as the proof in~\cite{Eichler-Zagier} actually
  shows.)  In addition, one has to use that a fundamental domain for
  $\SL^J\backslash \PH\times\C$ is given by $\{(\tau,x\tau+y):\tau\in
  D, x,y\ge 0, x+y\le 1\}$. Here $D$ is a fundamental domain for
  $\SL\backslash \PH$.  Note the condition $x+y\le 1$, which is due to
  the action of the scalar matrix $-1$ in~$\SL$, which gives
  $(-1)(\tau,z) = (\tau,-z)$. This gives another factor of $\frac12$,
  so that the correct factor in Thm.~5.3 becomes in fact
  $1/{4\sqrt{m}}$ (instead of $1/{\sqrt{2m}}$).

  From~\eqref{eq:psp-compatibility} and~\cite[Prop.~1.5,
  p.23]{Bruinier} it follows that
  \begin{equation*}
    H(P_{\Delta,r})
    =
    \alpha_{N,k,\Delta,r} P^L_{r/2N + \Z, |\Delta|/4N}
    ,
  \end{equation*}
  where $P^L_{r/2N + \Z, |\Delta|/4N}$ is the Poincar\'e series as
  introduced in~\cite[Eq.~(1.11), p.19]{Bruinier}.  The Fourier
  expansion of $P^L_{r/2N + \Z, |\Delta|/4N}$ has been computed
  in~\cite[Thm.~1.4, p.19]{Bruinier}.  The theorem is now an immediate
  consequence of the formula loc.~cit.
\end{proof}

\subsection{The Weil representation}
\label{sec:Weil-representation}

The last requirement to make our formulas for the Petersson scalar
products of Jacobi forms explicit and ready for implementation is the
determination of the matrix elements of the Weil representation $\rho$
which is defined by the formulas~\eqref{eq:Weil-representation}.
Explicit formulas are given e.g.~in~\cite[Prop.~1.6]{Shintani}.
However, a direct implementation of these formulas would involve, for
$A=\mat abcd$ with $c\not=0$, a sum of length $|c|$. A trivial
estimate shows that with this formula the number of exponential
evaluations necessary to compute $p_{\Delta,r}\left(\Delta',r'\right)$
is of the order of magnitude $M^{3}$ if we truncate the outer
sum~\eqref{eq:p_D_r_Dp_rp} at~$\left|c\right|=M$.  The formula of
Theorem~\ref{thm:Weil-representation} below which we will use is
computationally less expensive since it requires only one exponential
evaluation instead of a sum of $|c|$ such evaluations.

For stating our formula we need some notations. To each matrix $A =
\mat abcd$ in $\SL$ we associate an element $A^*$ in $\Mp$ by setting
$A^* := (A,\sqrt{c\tau+d})$. (Note that $A\mapsto A^*$ does not define
a homomorphism of groups.) Here and in the sequel, for a complex
number $w$, we use $\sqrt w$ for its root which is situated in the
right half plane but not on the negative imaginary axes. We write
$e^n(w)$ and $e_n(w)$ for $e^{2\pi i nw}$ and $e^{2\pi i w/n}$,
respectively.  For $h$, $h'$ in $\frac1{2N}\Z/\Z$ we define the matrix
element $\rho(A^*)_{h,h'}$ by the equation
\begin{equation*}
  \rho(A^*)\,e_{h'} = \sum_{h\in\frac 1{2N}\Z/\Z} \rho(A^*)_{h,h'}\,e_{h}
  .
\end{equation*}
Finally, for a nonzero rational number $r = \frac ba$
($\sym{gcd}(a,b)=1$, $a>0$) we define $\sym{signature}(r)$ as the
integer modulo 8 such that
\begin{equation*}
  e_8(\sym{signature}(r))
  =
  \left(\frac ba \right) e_8(1-a_1 + \sym{even}(a)\,ba_1)
  ,
\end{equation*}
where $\sym{even}(a)$ equals 1 or 0 accordingly as $a$ is even or odd,
where we write $a=2^na_1$ with odd $a_1$, and where $\left(\frac ba
\right)$ denotes the generalized Legendre symbol (which is
multiplicative in $a$ and in $b$, and which satisfies $\left(\frac b2
\right)=1$ if $b\equiv \pm 1\bmod 8$ and $\left(\frac b2 \right)=-1$
if $b\equiv \pm 3\bmod 8$.
\begin{Lemma}
  Let $r$ be a nonzero rational number, $a$ its (positive)
  denominator, and suppose that $a$ is not exactly divisible by $2$.
  Then
  \begin{equation*}
    \frac 1{\sqrt{\tilde a}}
    \sum_{x \bmod \tilde a} e(rx^2)
    = 
    e_8\big(\sym{signature}(2r)\big)
    ,
  \end{equation*}
  where $\tilde a = a/\sym{gcd}(2,a)$.
\end{Lemma}
\begin{proof}
  This identity can be easily verified using the well-known formulas
  \begin{equation*}
    \sum_{x\bmod a}e_a\big(bx^2\big) =
    \begin{cases}
      \leg ba \sqrt{a}
      & \text{ if } a \equiv 1 \bmod 4 \\
      \leg ba i \sqrt{a}
      & \text{ if } a \equiv 3 \bmod 4\\
      \leg b{2a} e^{2\pi i b/8} \sqrt{2a}
      & \text{ if  $a=2^n$, $n\geq 2$}.\\
    \end{cases}
  \end{equation*}
  for standard Gauss sums~\cite{Gauss}, and substituting $x=2^nt+a_1u$
  in a Gauss sum for a general $a=2^na_1$ (with odd $a_1$), where $t$
  and $u$ run through a complete set of representatives modulo $a_1$
  and $2^n$, respectively.
\end{proof}

\begin{Theorem}
  \label{thm:Weil-representation}
  For $A = \mat abcd$ in $\SL$ and integers $x$, $x'$, the matrix
  element~$\rho(A^*)_{h,h'}$ ($h=x/2N+\Z$, $h'=x'/2N+\Z$) of the
  representation $\rho$ of $\Mp$ defined
  by~\eqref{eq:Weil-representation} is given by the following formula:
  \begin{equation*}
    \rho(A^*)_{h, h'}
    =
    \chi(A)\sqrt{\frac{\sym{gcd}(2N,c)}{2N}}\,
    e_{4N}\left(abx^2+2bcxy+cdy^2 + 2z_c(bx+dy) + abz_c^2\right)
  \end{equation*}
  if there exists a $y$ such that $x' \equiv ax+cy+z_c \bmod 2N$ (and
  then the right hand of this formula does not depend on the choice of
  $y$), and $\rho(A^*)_{h, h'}=0$ otherwise.  Here $z_c=N$ if $c/2N$
  is a 2-adic unit, and $z_c=0$ otherwise.  Finally,
  \begin{equation*}
    \chi(A)
    =
    e_8\left(- \sym{sign}(c)\right)\cdot
    \begin{cases}
      e_8\big(\sym{signature}(aN/2c)\big) &\text{if $c/2N$ is a 2-adic unit}\\
      e_8\big(\sym{signature}(2aN/c)\big) &\text{otherwise}.
    \end{cases}
  \end{equation*}
  if $c\not=0$ and $\chi(A)=1/\sqrt{\sym{sign}(d)}$ otherwise.
\end{Theorem}
\begin{proof}
  The claimed formula is proved in~\cite[Sec.~1.4]{Skoruppa1} (see
  formulas~(26), (26)' and (27) on p.~37 loc.cit.), and in a slightly
  different form in~\cite[Prop.~4.1]{Skoruppa-Zagier2}. Note that the
  formulas in~\cite{Skoruppa1} refer to a right action
  $(\vartheta,\alpha)\mapsto \vartheta|_{\frac12,N}\alpha$ of $\Mp$ on
  a certain space $Th_N$ spanned by theta functions
  $\vartheta_{N,\rho}$ ($\rho \bmod 2N$). The precise definitions of
  these objects is not important here. For applying the formulas
  in~\cite{Skoruppa1} one has to identify $\C[\frac1{2N}\Z/\Z]$ with
  the dual of the space $Th_N$ via $e_{h}(\vartheta_{N,\rho})=1$ if
  $h=\rho/2N + \Z$ and $e_{h}(\vartheta_{N,\rho})=0$ otherwise. It is
  then easily verified (using~\cite[Satz~0.1, p.~10]{Skoruppa1}) that
  $\big(\rho(\alpha)\,e_h\big)(\vartheta) =
  e_h\big(\vartheta|_{\frac12,N}\alpha\big)$ for all $h$ in
  $\frac1{2N}\Z/\Z$ and all $\vartheta$ in $Th_M$. Inserting for
  $\vartheta|_{\frac12,N}\alpha$ the formulas (26), (26)' and (27)
  of~\cite[Sec.~1.4]{Skoruppa1} proves the claimed formulas apart from
  the fact that the quantity $\chi(A)$ is given in~\cite{Skoruppa1} in
  the form
  \begin{equation*}
    \chi(A)
    =
    \begin{cases}
      1/\sqrt{\sym{sign}(d)}
      &c=0\\
      \frac 1{\sqrt{\sym{gcd}(2N,c)ci}}\sum_{\nu\bmod
        2c}e_{4c}(aN\nu^2)
      &\text{if $c/2N$ is a 2-adic unit}\\
      \frac1{\sqrt{\sym{gcd}(2N,c)ci}}\sum_{\nu\bmod c}e_c(aN\nu^2)
      &\text{otherwise}.
    \end{cases}
  \end{equation*}
  (Note that formula (27) contains a typo stating that $\chi(a)=1$ for
  $c=0$, which, however, can be easily corrected by noticing that
  $\vartheta_{N,\rho}|_{\frac12,N}\mat {-1}00{-1} =
  i^{-1}\vartheta_{N,-\rho}$.)  By the Lemma we can rewrite $\chi(A)$
  in the given form. This proves the theorem.
\end{proof}

\subsection{Implementation}

All the computations were done using Sage~\cite{Sage}.  The routines
for interval arithmetic are implemented in Sage via the package {\tt
  mpmath}.  The code for computing the Weil representation, the
Poincar\'e series and finally the files containing the necessary Gram
matrices $\Gamma_h$ can be found on~\cite{Data}.  The numerical
computation of the Dirichlet series $D_{F,G}(s)$, which needs in
addition to the Gram matrices the Fourier coefficients of the
eigenforms in question is done via a short Sage script (also to be
found loc.cit.), which in turn calls the Siegel modular forms
package~\cite{Ryan-Skoruppa} mentioned in
Section~\ref{sec:introduction}.

\subsection{Discussion of error estimates}
\label{error-estimates}
For the computation of the Gram matrix~$\Gamma_h$ of
Theorem~\ref{thm:scalar-product} we need to compute the Fourier
coefficients $p_{\Delta,r}\left(\Delta',r'\right)$ of
Theorem~\ref{thm:Poincar-series}.  Using the standard asymptotics of
the J-Bessel function
$J_{v}\left(x\right)\sim\frac{1}{\Gamma\left(v+1\right)}$
$\left(\frac{x}{2}\right)^{v+1}$ for~$x\ll1$ and the fact that
$|\rho\left(M\right)_{h,h'}|\le 1$ it follows that in order to obtain
an error of order $\varepsilon=10^{-T}$ for the computation of a
coefficient we need to truncate the sum over $c$ at about
$c=M\left(T\right)$ with
\begin{equation*}
  M\left(T\right)>O\left(1\right)10^{\frac{T}{k-\frac{3}{2}}}
\end{equation*}
with a constant (which can be made explicit) depending on
$k,\Delta,\Delta'$ and $N$.  From this we see that for small values of
$k$, in particular, this evaluation can be very time consuming while
for higher weight $k$ we can truncate the sum earlier. At this point
one should also mention that the dimension of the space $S_{k,N}$,
i.e. the number of necessary Poincar\'e series to consider, increases
linearly in $k$ and in $N$ (see~\cite[Thm.~1]{Skoruppa-Zagier} for the
dimension formula), so that the number of coefficients
$p_{\Delta,r}\left(\Delta',r'\right)$ to compute grows quadratically
in $k$ and in $N$.  In Table~\ref{tbl:Jacobi-dimensions} we list the
dimensions of the spaces of Jacobi cusp forms $S_{k,N}$ which fall
within the range of our computations.

We need to make sure that the set of Poincar\'e series $\left\{
  P_{h\left(i\right)}\right\} _{i=1,\ldots,n}$ for a given sequence of
indices $h=\left\{ \left(\Delta_{i},r_{i}\right)\right\} _{i} (n=\dim
S_{k,N})$ form a basis of $S_{k,N}$.  For this it suffices to verify
that the associated Gram matrix $\Gamma_{h}$ \eqref{eq:Gram-matrix} is
invertible, i.e. that its determinant is non-zero. Let
$\varepsilon=10^{-T}$ ($T\gg 1 $) and suppose that we use a
sufficiently high working precision. Let
$N\left(\Gamma_{h}\right)=\left(a_{ij}\right)$ be a numerical
approximation to $\Gamma_{h}$ with an error less than
$\varepsilon$. By this we mean that
$a_{ij}=N(C_{P_{h\left(i\right)}}\left(h\left(j\right)\right))$ and
$\max_{0\le i,j\le n}
\left|a_{ij}-C_{P_{h\left(i\right)}}\left(h\left(j\right)\right)\right|
< \varepsilon$.  It follows that
$-\varepsilon<a_{ij}-C_{P_{h\left(i\right)}}\left(h\left(j\right)\right)<\varepsilon$,
or equivalently
$C_{P_{h\left(i\right)}}\left(h\left(j\right)\right)\in
\left(a_{ij}-\varepsilon,a_{ij}+\varepsilon\right)=:I_{ij}$ for each
$i,j$.  Let $I\left(\Gamma_{h}\right)$ be the interval matrix with
intervals $I_{ij}$ as entries.  The matrix $I\left(\Gamma_{h}\right)$
is then also an approximation to $\Gamma_{h}$ in the sense that
$\Gamma_{h}\in I\left(\Gamma_{h}\right)$.  By elementary properties of
interval arithmetic (see e.g.~\cite{Alefeld-Herzberger}) one can see
that $\det\Gamma_{h}\in\det I\left(\Gamma_{h}\right)$ and it is
therefore enough to verify that $0\notin\det
I\left(\Gamma_{h}\right)$.  In practice, we start with one function
$P_{h\left(1\right)}$ and the corresponding one-by-one gram matrix,
$\Gamma_{h}^{\left(1\right)}$. We then choose a second function,
$P_{h\left(2\right)}$, and calculate the determinant (as an interval)
of the interval approximation to the partial Gram matrix of these two
functions. If this interval does not contain zero we add the function
$P_{h\left(2\right)}$ to the set and have a non-singular two-by-two
matrix $\Gamma_{h}^{\left(2\right)}$. This procedure then continues
until we have a set of $n$ linearly independent functions and a
non-singular matrix $\Gamma_{h}$.

It is not difficult to see that if the absolute error in (the infinity
norm of) of our approximation $N(\Gamma_{h})$ of $\Gamma_{h}$ is
$\varepsilon$ then the absolute error of our numerical approximation
$N(\psp {F_{N}}{G_{N}})$ of $\psp {F_{N}}{G_{N}}$ accordingly is also
bounded by a constant (depending on $F_N$ and $G_N$) times
$\varepsilon$.  To see this, first observe that in all our
applications the functions $F_{N}$ and $G_{N}$ are known exactly (if
defined over a number field other than $\mathbb{Q}$ we can approximate
them to arbitrarily small precision).  For any $\varphi,\psi\in
S_{k,N}$ let $\psp{\varphi}{\psi}_{N \left(\Gamma_{h}\right)}
:=\gamma_{h}\left(\phi\right)^t
N\left(\Gamma_{h}\right)\overline{\gamma_{h}\left(\psi\right)}$ denote
the inner-product with respect to the
matrix~$N\left(\Gamma_{h}\right)$. Then
\begin{align}
  \label{eq:abserr}
  \left|\psp{\varphi}{\psi} - \psp \varphi{\psi}_{\tilde{\Gamma}_{h}}
  \right| & = |\gamma_{h} \left(\varphi\right)^t
  \left(\Gamma_{h}-\tilde{\Gamma}_{h}\right) \overline{\gamma_{h}
    \left(\psi\right) }
  |\nonumber\\
  & \le \Vert \Gamma_{h}-\tilde{\Gamma}_{h} \Vert_{\infty} \Vert
  \gamma_{h}\left(\varphi\right) \Vert_{2} \Vert
  \gamma_{h}\left(\psi\right) \Vert_{2} .
\end{align}
Thus the absolute error depends on the $l^{2}$-norms of the vectors
$\gamma_{h}\left(\varphi\right)$ and $\gamma_{h}\left(\psi\right)$.
Since the numbers $N(\psp{F_N}{G_N})$ tend to grow rapidly with $N$ it
makes more sense to consider, as a measure of accuracy, the number of
correct digits, i.e. the relative error, instead of the absolute
error.  Unfortunately we are not able to estimate $\psp\varphi\psi$
from below and thus can not give a corresponding ``predictive'' bound
for the relative error in $N(\psp{F_N}{G_N})$.  However, it is clear
that if
\begin{equation*}\left| N(\psp{F_N}{G_N}) - \psp{F_N}{G_N}\right| <
\varepsilon \cdot \Vert \gamma_{h} \left(F_N\right) \Vert_{2} \Vert
\gamma_{h} \left(G_N\right) \Vert_{2}\le \varepsilon \cdot K
\end{equation*}
and $\left|\varepsilon\cdot K\right| < \left|N(\psp{F_N}{G_N})\right|$
then the relative error can be estimated by
\begin{equation*}
\frac{\left| N(\psp{F_N}{G_N}) - \psp{F_N}{G_N}\right|
}{\left|\psp{F_N}{G_N}\right|} < \frac {\varepsilon\cdot
  K}{\left|\left|N(\psp{F_N}{G_N})\right| - \varepsilon \cdot K
  \right|}
\end{equation*}
and we are therefore able to give a precise a posteriori bound for the
relative error. Provided that $N\left(\psp{F_N}{G_N}\right)$ is
nonzero (which turns out to be the case in all cases we considered so
far) we are always able to go back and decrease the initial
$\varepsilon$ until we have achieved any desired bound.  Since
rigorous, or even validated, numerics never give insurance against
e.g. errors in the actual implementation we also felt compelled to
apply our algorithms to a case where we know the results explicitly.

Recall that if $G$ is in the Maass spezialschaar and $F$ is an
eigenform, then $D_{F,G}(s) = \psp{F_1}{G_1} L(F,s) L(F,s)$. To verify
that the Gram data is correct as well as that the relative error
behaves well, we compare the computed values for the coefficients of
$D_{F,G}(s)$ with those of $L(F,s)$ (see~Tables~\ref{tbl:comparison}
and \ref{tbl:numerics}).  The coefficients of $L(F,s)$ can be computed
exactly using~\cite{Ryan-Skoruppa}, It seems that in most cases the
relative error can be estimated from above by a constant (not growing
very fast with respect to $N$) times $\Vert
\Gamma_{h}-\tilde{\Gamma}_{h}\Vert_{\infty}$. That is, the quotient of
$N(\psp{F_N}{G_N})$ divided by $\Vert \gamma_{h} \left(F_N\right)
\Vert_{2} \Vert \gamma_{h} \left(G_N\right) \Vert_{2}$ does not seem
to tend to zero too fast with $N$.

Another way we verify that the Gram data is correct is we check, as
in~\cite{Farmer-Ryan}, that the series $D_{FG}$ for $F$ and $G$ not in
spezialschaar, satisfies the functional equation shown
in~\cite{Kohnen-Skoruppa}.

\subsection{Runtimes}

To give a flavor of the amount of time needed to run an actual
computation of the Gram matrices we can consider weight $k=30$ where
we spent slightly less than 4h (14254s) on a 2666MHz processor to
compute the Gram matrices for $1\le N\le20$ with an error bound for
the Poincar\'e series of $10^{-50}.$ This involved computing a total
of 16340 Fourier coefficients.  We note that the time for a single
Poincar\'e series computation to a given precision is essentially
bounded by a constant independent of~$k$ and~$N$. Thus the runtime
only increases with the number of coefficients required.  The
dependence of the runtime on the precision requested is slightly more
complicated.  For the proof of the Main Theorem in the case $k=30$ we
needed a precision $10^{-90}$ and the corresponding Poincar\'e series
computations (with a working precision of $631$ digits) took more than
two weeks on the same system as mentioned above.

\section{Proof of the Main Theorem and its Supplement}
\label{sec:proof}

Fix a weight $k$ (between $20$ and $30$), and let $F_1$,\dots, $F_d$
be a basis of common Hecke eigenforms of the full space $M_k$ of
Siegel modular forms of weight $k$. We assume that $F_1$,\dots, $F_h$
are the interesting ones, i.e.~that they span the space $S_k^?$. We
want to show that the series $D_{F_i,F_j}(s)$ and $L(F_l,s)$ ($1\le i
\le j \le h$, $1 \le l \le d$) are linearly independent. For this we
set
\begin{equation*}
  M
  :=
  \left(
    \begin{smallmatrix}
      \psp{F_{1,1}}{F_{1,1}}/\gamma_{1,1}&\hdots&\psp{F_{1,n}}{F_{1,n}}/\gamma_{1,1}\\
      \psp{F_{1,1}}{F_{2,1}}/\gamma_{1,2}&\hdots&\psp{F_{1,n}}{F_{2,n}}/\gamma_{1,2}\\
      \vdots&&\vdots\\
      \psp{F_{e,1}}{F_{e,1}}/\gamma_{e,e}&\hdots&\psp{F_{e,n}}{F_{e,n}}/\gamma_{e,e}\\
      \lambda_{F_{1,1}}&\hdots&\lambda_{F_{1,n}}\\
      \vdots&&\vdots\\
      \lambda_{F_{d,1}}&\hdots&\lambda_{F_{d,n}}
    \end{smallmatrix}
  \right)
  .
\end{equation*}
Here $n=\frac{e(e+1)}2 + d$, we use $F_{i,N}$ for the $N$-th
Fourier-Jacobi coefficient of $F_i$, and we use
$\gamma_{i,j}=\psp{F_{i,1}}{F_{j,1}}$. The $\lambda_{F_{j,l}}$ are the
eigenvalues of the $F_j$ with respect to $T(l)$
(see~\eqref{eq:spinor-l-function}). We want to show $\det M \not=0$.
The eigenvalues $\lambda_{F_{j,l}}$ are totally real numbers and can
be computed exactly~\cite{Skoruppa2}. Assume we approximate them
numerically by rational numbers $N(\lambda_{F_{j,l}})$ such that
\begin{equation*}
  |\lambda_{F_{j,l}} - N(\lambda_{F_{j,l}})|<\epsilon
  .
\end{equation*}
Furthermore, using the method described in
Section~\ref{sec:computation} we compute rational approximations
$N((\psp{F_{i,l}}{F_{j,l}}/\gamma_{i,j})$ such that
\begin{equation}
  \label{eq:errors}
  |\psp{F_{i,l}}{F_{j,l}} - N(\psp{F_{i,l}}{F_{j,l}})|/|\gamma_{i,j}|<\epsilon
  .
\end{equation}
(The numerical calculations show that indeed $\gamma_{i,j}\not=0$ for
all $i,j$.)  If we denote by $N(M)$ the matrix obtained from $M$ by
replacing the coefficients $m_{p,q}$ by their respective rational
approximations $N(m_{p,q})$ we conclude that $\Vert M-N(M)\Vert_\infty
< \epsilon$.  As in section \ref{error-estimates} we construct an
interval matrix approximation to $M$, $I(M)$, with intervals
$I_{p,q}=\left(N\left(m_{p,q}\right)-\epsilon,N\left(m_{p,q}\right)+\epsilon\right)$
as entries. Since $M \in I(M)$ we are once again able to show that the
determinant of $M$ is non-zero by verifying that the determinant $D$
(which is an interval) of $I(M)$ does not contain zero.
If we compute the gram matrix $\Gamma_h$ to the precision $\epsilon_1$
we obtain an upper bound $\epsilon_2$ for the relative error
$\epsilon$ in \eqref{eq:errors} by
\begin{equation*}
  \epsilon_2 = \epsilon_1 \cdot \frac{\max_{i,j}{\Vert \gamma_h(F_{i,j})
      \Vert^{2}_2} }{\min_{i,j}\left| \gamma_{i,j} \right|}.
\end{equation*}
It turns out that by choosing $\epsilon_1$ as in
Table~\ref{tbl:epsilons} we obtained an $\epsilon_2$ such that the
determinant of the corresponding interval matrix $I(M)$ did not
contain zero.  This proves the Theorem.

For the actual computed values of upper and lower bounds of $\Vert
\gamma_h(F_{i,j}) \Vert_2$ and $\left| \gamma_{i,j} \right|$,
respectively, see Table \ref{tbl:numerics}.  The values of
$\epsilon_2$ obtained from these bounds, as well as the determinant
$D=\left( D_m - D_\delta, D_m + D_\delta\right)$ are given in
Table~\ref{tbl:epsilons}.  That zero is not contained in either of
these intervals is easily checked.  It should also be remarked that
for all our computations for a given $k$, we used a working precision
set high enough to represent $D$ with an absolute error comparable to
$\epsilon_2$.  As an example we computed the Poincar\'e series for
$k=30$ using $631$ digits precision.

The supplement was proved in a similar manner, using interval
arithmetic and the error bounds~\eqref{eq:abserr}.

\section{Concluding Remarks}
\label{sec:numerical-results}

We computed $D_{F,G}(s)$ numerically following the method of
Section~\ref{sec:computation} for those cusp forms $F$ and $G$ listed
by name in Table~\ref{tbl:Siegel-dimensions}.  The Fourier
coefficients of these Siegel cusp forms and the coefficients of their
spinor $L$-functions can be computed exactly using the
package~\cite{Ryan-Skoruppa} which implements the methods of
computation of~\cite{Skoruppa2}. As additional control for our
calculations, we verified that our numerical approximations of
$D_{F,G}$, for $F$ as in Table~\ref{tbl:Siegel-dimensions} and for a
Maass spezial form $G$, equals indeed (within our error bounds) the
spinor $L$-function of $F$, as predicted by the second theorem
in~\cite{Kohnen-Skoruppa} (see also Tables \ref{tbl:comparison} and
\ref{tbl:numerics}).

Andrianov asks~\cite[Concl.~2, p.~114]{Andrianov} if all Dirichlet
series satisfying the functional equation \eqref{eq:functional} are
linear combinations of the spinor $L$-series $L(H,s)$ where~$H$ runs
through the eigenforms in the full space $M_k$ of Siegel modular forms
of weight~$k$ on the full Siegel modular group.  We can answer this
question in the negative.  Namely, we computed all the Hecke
eigenforms in weights 20, 22, 24, 26, 28 and~s30 and their
corresponding $L$-series; in particular we computed the $L$-series for
the Eisenstein series, Klingen-Eisenstein series, Maass lifts, and
interesting forms that are Hecke eigenforms.  For all of the cases
listed in Table~\ref{tbl:Siegel-dimensions}, $D_{F,G}(s)$ was not a
linear combination of spinor $L$-series.

This raises the question of the nature of the series $D_{F,G}$ for
interesting forms $F$ and $G$. It is natural to ask first of all
whether they are $L$-functions at all in the sense of Selberg,
i.e.~whether they belong to the Selberg class. In fact, they do not,
since they do not possess an Euler product expansion (see the
supplement to the main theorem and Table~\ref{tbl:coeffs}). If we set
\begin{equation}
  \label{eq:D-tilde}
  \widetilde D_{F,G}(s)
  = 
  \psp{F_1}{G_1}^{-1}
  D_{F,G}(s+k-\frac32)
  .
\end{equation}
(the scalar product $\psp {F_1}{G_1}$ is different from $0$ in all
cases covered by Table~\ref{tbl:Siegel-dimensions}) the functional
equation~\eqref{eq:functional} becomes
\begin{equation*}
  \widetilde D_{F,G}^*(s)
  :=
  (2\pi)^{-2s-2k+3}
  \Gamma(s+k-\frac32)
  \Gamma(s+\frac12)
  \widetilde D_{F,G}(F,s)
  =\widetilde D_{F,G}^*(1-s)
  .
\end{equation*}
Accordingly to~\cite[Theorem~1]{Kohnen-Skoruppa} the function
$\widetilde D_{F,G}(s)$ has a pole at $s=\frac32$. This shows again
that $\widetilde D_{F,G}(s)$ is not in the Selberg class.

\clearpage
\section{Tables and Figures}
\label{sec:tables}

\begin{table}[ht]
  \begin{center}
    \caption{Dimensions of subspaces of $M_k$.  (KE =
      Klingen-Eisenstein series, MC = Maass spezialschaar cusp forms,
      Eigenforms = names of Galois orbits of interesting eigenforms
      following the naming conventions in~\cite{Skoruppa2})}
    \label{tbl:Siegel-dimensions}
    \begin{tabular}{|c||lllll|}
      \hline
      $k$   &   $S_k$  &   KE  &    MC   &   $S_k^?$ & Eigenforms\\
      \hline\hline
      20 & 5 & 2 & 2 & 1  & $\Upsilon_{20}$\\
      22 & 6 & 2 & 3 & 1  & $\Upsilon_{22}$\\
      24 & 8 & 3 & 3 & 2  & $\Upsilon_{24a}$, $\Upsilon_{24b}$\\
      26 & 7 & 2 & 3 & 2  & $\Upsilon_{26a}$, $\Upsilon_{26b}$\\
      28 & 10 & 3 & 4 & 3 & $\Upsilon_{28}$\\
      30 & 11 & 3 & 4 & 4 & $\Upsilon_{30}$\\
      \hline
    \end{tabular}
  \end{center}
\end{table}

\begin{table}[ht]
\tabcolsep=3pt
  \begin{center}
    \caption{Dimensions of the spaces $S_{k,N}$ for even weights $20
      \le k \le 30$ and indices $1\le N \le 20$.}
    \label{tbl:Jacobi-dimensions}
    \begin{tabular}{|r||rrrrrrrrrrrrrrrrrrrr|}
      \hline
      & 1 & 2 & 3 & 4 & 5 & 6 & 7 & 8 & 9 & 10 & 11 & 12 & 13 & 14 & 15 & 16 & 17 & 18 & 19 & 20 \\
      \hline\hline
      20 & 2 & 4 & 5 & 6 & 8 & 9 & 11 & 12 & 13 & 15 & 17 & 17 & 20 & 21 & 22 & 23 & 26 & 26 & 29 & 29 \\
      22 & 3 & 4 & 6 & 7 & 9 & 10 & 13 & 13 & 15 & 17 & 19 & 19 & 23 & 23 & 25 & 26 & 29 & 29 & 33 & 32 \\
      24 & 3 & 5 & 7 & 8 & 10 & 12 & 14 & 15 & 17 & 19 & 21 & 22 & 25 & 26 & 28 & 29 & 32 & 33 & 36 & 36 \\
      26 & 3 & 5 & 7 & 8 & 11 & 12 & 15 & 16 & 18 & 20 & 23 & 23 & 27 & 28 & 30 & 31 & 35 & 35 & 39 & 39 \\
      28 & 4 & 6 & 8 & 10 & 12 & 14 & 17 & 18 & 20 & 23 & 25 & 26 & 30 & 31 & 33 & 35 & 38 & 39 & 43 & 43 \\
      30 & 4 & 6 & 9 & 10 & 13 & 15 & 18 & 19 & 22 & 24 & 27 & 28 & 32 & 33 & 36 & 37 & 41 & 42 & 46 & 46 \\
      \hline
    \end{tabular}
  \end{center}
\end{table}

\begin{table}[ht]
  \begin{center}
    \caption{For the Maass lift $G$ of the first Fourier-Jacobi
      coefficient $\phi$ of $\Upsilon24a$, we compute numerically the
      coefficients $a(N)$ of $D_{\Upsilon24a,G}(s)/\psp{\phi}{\phi}$
      by the method in Section~\ref{sec:computation} and we compute
      exactly the coefficients $b(N)$ of $L(\Upsilon24a,s)$ (which are
      integers).  This is an additional check that our method of
      computation gives correct results since a theorem
      in~\cite{Kohnen-Skoruppa} tells us that, in this case, these two
      Dirichlet series coincide.  }
    \label{tbl:comparison}
    \begin{tabular}{|c||l|l|}\hline
      $N$ & $a(N)$ & $b(N)$\\
      \hline\hline
      1 & $\hphantom{-}1.0000000000000000000000000000 $    &    $\hphantom{-} 1$\\
      2 & $-5.5603200000000000000000000000 \cdot10^{6}$  &     $-5560320$\\
      3 & $-5.3017924680000000000000000000\cdot 10^{10}$ &     $-53017924680$\\
      4 & $\hphantom{-}4.3592282275840000000000000000\cdot 10^{12}$  &     $\hphantom{-}4359228227584$\\
      5 & $-3.3324163624500000000000000000\cdot 10^{13}$ &     $-33324163624500$\\
      6 & $\hphantom{-}2.9479662695669760000000000000\cdot 10^{17}$  &     $\hphantom{-}294796626956697600$\\
      7 & $\hphantom{-}8.9548405531223548000000000000\cdot 10^{18}$  &     $\hphantom{-}8954840553122354800$\\
      8 & $-1.7002266475219648512000000000\cdot 10^{20}$ &     $-170022664752196485120$\\
      9 & $\hphantom{-}3.3581886607436193369000000000\cdot 10^{19}$  &     $\hphantom{-}33581886607436193369$\\
      10& $\hphantom{-}1.8529301348457984000000000000\cdot 10^{20}$  &     $\hphantom{-}185293013484579840000$\\
      \hline
    \end{tabular}
  \end{center}
\end{table}

\begin{table}[ht]
  \begin{center}
    \caption{The first 18 coefficients of $\widetilde
      D_{\Upsilon24x,\Upsilon24y}(s)$ (see\eqref{eq:D-tilde}), where
      $x,y \in \{a,b\}$, where $\Upsilon24a$, $\Upsilon24b$ denote the
      two interesting eigenforms in weight 24.  Note that $\widetilde
      D_{\Upsilon24x,\Upsilon24y}(s)$ satisfies a functional equation
      with respect to $s\mapsto 1-s$. More data and this data with
      more accuracy are available at~\cite{Data}.}
    \label{tbl:coeffs}
    \begin{tabular}{|c||c|c|c|}\hline
      &    $\widetilde D_{\Upsilon24a,\Upsilon24a}(s)$ & $\widetilde D_{\Upsilon24a,\Upsilon24b}(s)$ & $\widetilde D_{\Upsilon24b,\Upsilon24b}(s)$\\
      \hline\hline
      1  &  $1.0000000000\dots$ & $\hphantom{-}1.0000000000\dots$ &  $1.0000000000\dots$\\
      2  &  $1.2562996887\dots$ & $-1.6301379426\dots$ & $1.1807313893\dots$\\
      3  &  $1.7810603106\dots$ & $\hphantom{-}2.0303021423\dots$ &  $0.8677710052\dots$\\
      4  &  $2.0741433142\dots$ & $-3.4104218951\dots$ & $2.8372669032\dots$\\
      5  &  $2.8899783797\dots$ & $\hphantom{-}0.0053403337\dots$ &  $2.3422724376\dots$\\
      6  &  $2.6795441187\dots$ & $\hphantom{-}3.4130123501\dots$ &  $2.5830326424\dots$\\
      7  &  $3.6002676445\dots$ & $\hphantom{-}1.2502722224\dots$ &  $2.2422068960\dots$\\
      8  &  $3.0295878336\dots$ & $-0.2164387720\dots$ & $3.5760646077\dots$\\
      9  &  $3.9970248055\dots$ & $-1.9659335003\dots$ & $3.2031268531\dots$\\
      10  &  $2.9874427387\dots$ & $-2.2751217843\dots$ & $3.5506087520\dots$\\
      11  &  $3.5420885329\dots$ & $-0.5800680650\dots$ & $2.6657794870\dots$\\
      12  &  $2.9887639258\dots$ & $\hphantom{-}2.3742550783\dots$ &  $3.3541225360\dots$\\
      13  &  $4.7252631756\dots$ & $-1.4101441627\dots$ & $3.0042305971\dots$\\
      14  &  $3.6881022526\dots$ & $\hphantom{-}2.8085205331\dots$ &  $2.5295615640\dots$\\
      15  &  $3.2033969064\dots$ & $-2.1277658831\dots$ & $2.9089500309\dots$\\
      16  &  $4.8466409087\dots$ & $\hphantom{-}1.6172735165\dots$ &  $5.3475003622\dots$\\
      17  &  $5.5339720201\dots$ & $-2.0014367544\dots$ & $4.1580587188\dots$\\
      18  &  $4.5912959032\dots$ & $-2.3762191782\dots$ & $4.6382023273\dots$\\\hline
    \end{tabular} 
  \end{center}
\end{table}

\begin{table}[ht]
  \begin{center}
      \caption{Computed upper bounds for the $l^2$ norms of $\gamma_h\left(F_{i,j}\right)$ and lower bounds for the inner products $\gamma_{i,j}=\psp{F_{i,1}}{F_{j,1}}$.
The last column contains the maximum error in the comparison between the Hecke eigenvalues $\lambda_{F_{i,j}}$ and the numerical approximations, $\tilde{\lambda}_{F_{i,j}}$, 
  computed using the inner product of $F_i$ with a form from the Maass spezialschaar. 
In all cases $1\le i,j\le n$ where $n$, given in the second column, is the dimension of the space of $L$-series which we want to show is linearly independent. }
    \label{tbl:numerics}
    \begin{tabular}{|c||c|c|c|c|}
      \hline 
      $k$ &$n$ & $\max\left(\left\Vert \gamma_h\left(F_{i,j}\right) \right\Vert _{2}\right)$  
          & $\min \left| \gamma_{i,j}\right|$ 
          & $\max |\lambda_{F_{i,j}}-\tilde{\lambda}_{F_{i,j}} |$  \tabularnewline
      \hline
      \hline 
      20  & 5 &$1.2\cdot10^{09}$ & $1\cdot10^{1}\hphantom{{}^{-1}}$ &  $3.2\cdot10^{-45}$  \tabularnewline
      \hline 
      22 & 6 &$8.2\cdot10^{04}$ & $5\cdot10^{-11}$ &  $9.6\cdot10^{-56}$ \tabularnewline
      \hline 
      24 &  11&$7.8\cdot10^{09}$ & $1\cdot10^{-10}$ & $1.7\cdot10^{-55}$  \tabularnewline
      \hline 
      26 &  10&$5.2\cdot10^{10}$ & $1\cdot10^{-12}$ & $7.5\cdot10^{-72}$ \tabularnewline
      \hline 
      28 &  16&$8.8\cdot10^{16}$ & $1\cdot10^{0}\hphantom{{}^{-1}}$ &  $2.2\cdot10^{-71}$ \tabularnewline
      \hline 
      30 &  21&$2.9\cdot10^{12}$ & $3\cdot10^{-25}$ & $5.6\cdot10^{-96}$ \tabularnewline
      \hline
    \end{tabular}
  \end{center}
\end{table}

\begin{table}[ht]
  \begin{center}
    \caption{The values used for $\epsilon_1$ and the computed values of $\epsilon_2$
 as well as the computed values of the interval determinant of $I(M)$, $D=\left(D_m-D_\delta,D_m+D_\delta\right)$.}
    \label{tbl:epsilons}
    \begin{tabular}{|c||c|c|c|c|}
      \hline 
      $k$ & $\epsilon_{1}$ & $\epsilon_{2}$ & $D_{m}$ & $D_{\delta}$  \tabularnewline
      \hline
      \hline 
      20 & $1\cdot10^{-40}$ & $7.9\cdot10^{-26}$ & $-7.05\cdot10^{054}$ & $1.26\cdot10^{033}$ \tabularnewline
      \hline 
      22 & $3\cdot10^{-50}$ & $1.2\cdot10^{-23}$ & $-6.63\cdot10^{080}$ & $8.87\cdot10^{060}$ \tabularnewline
      \hline 
      24 & $2\cdot10^{-50}$ & $2.8\cdot10^{-13}$ & $\hphantom{-}7.38\cdot10^{209}$ & $1.63\cdot10^{206}$ \tabularnewline
      \hline 
      26 & $7\cdot10^{-65}$ & $4.0\cdot10^{-27}$ & $-6.80\cdot10^{190}$ & $1.27\cdot10^{173}$ \tabularnewline
      \hline 
      28 & $1\cdot10^{-65}$ & $4.8\cdot10^{-33}$ & $-3.21\cdot10^{399}$ & $6.45\cdot10^{384}$  \tabularnewline
      \hline
      30 & $1\cdot10^{-90}$ & $2.9\cdot10^{-25}$ & $\hphantom{-}1.53\cdot10^{607}$ & $1.13\cdot10^{607}$  \tabularnewline
      \hline
    \end{tabular}
  \end{center}
\end{table}

\clearpage
\bibliography{lfg} \bibliographystyle{amsplain}

\end{document}